\documentclass{amsart}
\usepackage{amsfonts}
\usepackage{amsmath,amssymb}	
\usepackage{amsthm}
\usepackage{amscd}
\usepackage{graphics}
\usepackage{graphicx}
{
\theoremstyle{definition}
\newtheorem{Def}{Definition}
\newtheorem{Ex}{Example}
\newtheorem{Rem}{Remark}

}

\newtheorem{Prop}{Proposition}
\newtheorem{Thm}{Theorem}

\begin{document}
\title[Inverse images of smooth PL maps]{Smooth maps compatible with simplicial structures and inverse images}

\author{Naoki Kitazawa}
\address{Industry of Mathematics for Industry, Kyushu University, 744 Motooka, Nishi-ku Fukuoka 819-0395, Japan}
\email{n-kitazawa.imi@kyushu-u.ac.jp}
\maketitle
\begin{abstract}
As higher dimensional versions of geometric studies of Morse functions, there have been various studies of smooth manifolds using generic smooth maps. As fundamental results, in these studies, they have found that inverse
 images of such maps often restrict the homotopy, topology and differentiable structures of the source manifolds. For example, if generic maps such that there exist inverse images of regular values being not
  null-cobordant, then the top-dimensional homology groups of their {\it Reeb} spaces, which are defined as the spaces of all connected components of inverse images of the maps, polyhedra of dimensions equal to those of the target spaces in considerable cases, and fundamental and important tools in the theory of generic smooth maps having some or sometimes much information of homology and homotopy groups etc., are shown to be non-trivial by Hiratuka and Saeki in 2012--4. In these cases, for example, we can see that the corresponding homology groups of the source manifolds are not trivial except several special cases. 
  
  In this paper, we show an extended result where there exists a connected component which is not zero considering a corresponding element of a module defined by considering suitable equivalence classes of smooth, connected and closed manifolds of a fixed dimension. The module is defined as a generalization of a cobordism group and the relation is defined as that of a cobordism relation. We also apply the results to several specific cases.
\end{abstract}

\section{Introduction, terminologies and notation}
\label{sec:1}

As a branch of the singularity theory of differentiable maps and its application to geometry of manifolds, higher dimensional versions of geometric studies of
 Morse functions on smooth manifolds have been
  actively studied. In this stream, geometric theory of generic maps such as {\it {\rm (}topologically{\rm )} stable} maps, defined as maps such that by slight perturbations the global topological properties of the maps are invariant, has developed; see \cite{golubitskyguillemin} for generic smooth maps and their singular points.
  
As a branch of such studies, generic maps having good
 geometric properties and geometry of manifolds admitting such maps have
  been studied. As a class of generic smooth maps, generic maps such
   that the inverse images of regular values, which are
    smooth and closed submanifolds of the source manifolds, are
     simple, have been studied. For example, theorems stating
      that source manifolds bound nice compact manifolds obtained by observing
       inverse images as
        investigated in \cite{saeki}, \cite{saeki3}, \cite{saekisuzuoka} and later in \cite{kitazawa} etc.. Note that
 these facts are, in some specific cases, essentially equivalent to the fundamental principle of {\it shadows} of $3$ and $4$-dimensional manifolds. {\it Shadows} are $2$-dimensional polyhedra realized as the retracts of the $4$-dimensional compact orientable manifolds whose non-empty boundaries are closed and orientable $3$-dimensional manifolds and representing these $3$ and $4$-dimensional manifolds well (see \cite{turaev} and see also \cite{costantino} for example), and motivated by this similarity, articles related to both of these two types of ideas have been published since 2000s such as \cite{costantinothurston} and \cite{ishikawakoda}. 

As another work, which is also a key ingredient in the main theorem of this paper, in \cite{hiratukasaeki2}, Hiratuka and Saeki found a topological constraint, given by
            a generic map such as a {\it topologically stable} map such that the inverse image of a regular value includes a null-cobordant connected component, to the {\it Reeb space} of the map. 

The {\it Reeb space} of a map is defined as
            the space of all connected components of inverse images (see \cite{reeb} and \cite{sharko} for example). Reeb spaces are polyhedra of dimensions equal to those of the target spaces in considerable cases, inherit fundamental invariants of source manifolds such as homology groups etc. in suitable cases, and are convenient stuffs for studying about algebraic and differential topological properties of source manifolds. Moreover, for example, a Reeb space is essentially regarded as a shadow introduced before in the case where
 the map is from a $3$-dimensional closed and orientable manifold into the plane. Last, to explain precisely about the result of Hiratuka and Saeki, the top dimensional homology groups of the Reeb spaces do not vanish in these situations and
 except several cases the corresponding homology groups of source manifolds do not vanish.

 In this paper, we show a statement similar to the result by Saeki and Hiratuka in the case where the inverse image of a regular value includes a component which is not zero as an element
 of a module consisting of equivalence classes regarded as generalizations of cobordism groups.    

This paper is organized as the following.

In section \ref{sec:2}, we review the {\it Reeb spaces} of continuous maps and {\it triangulable} maps. For example, {\it {\rm (}topologically{\rm )}stable} maps are important generic
 smooth maps and also triangulable (see \cite{hiratukasaeki} and \cite{shiota}). We introduce some
  explicit triangulable maps including (topologically) stable maps in Example \ref{ex:1}: several explicit cases where Reeb spaces inherit homology groups of source manifolds are presented. 
  
In section \ref{sec:3}, we introduce a cobordism-like module as an extension of a cobordism group observing
  differentiable structures of manifolds (in the oriented category) and
 neighborhood relations of inverse images of regular values of a smooth triangulable map.

In section \ref{sec:4}, we extend the mentioned result \cite{hiratukasaeki2} in the case where an inverse image contains a component which is not zero in a cobordism-like module
 defined in the previous section. By obtaining a chain
 representing a non-zero cycle of the Reeb space with the coefficient group being the cobordism group, they prove that the homology group does not vanish and we prove a similar result as Theorem \ref{thm:1}, which is a main theorem of the present paper, in a similar manner.

In the last section, we apply the results to explicit maps to which we cannot apply the original result. Obtained theorems in this section are also main theorems.

In this paper, manifolds and maps between them are smooth and of class $C^{\infty}$ unless otherwise stated. The {\it singular set} of a smooth map is defined as the set of all singular points of the map, the {\it singular value set} of the map is defined as
 the image of the singular set and the {\it regular value set} of the map is the complement of the singular value set.
  
Throughout this paper, $M$ is a smooth closed manifold of dimension $m \geq 1$, $N$ is a smooth manifold of dimension $n$ without boundary satisfying
  the relation $m \geq n \geq 1$, and $f$ is a smooth map from $M$ into $N$.
   Moreover, we denote the singular set of the map by $S(f)$.

The author would like to thank for Osamu Saeki for advising him to present us a main result or Theorem \ref{thm:2} of this paper when the author first explained this to him and owing to the advice, as an additional work, the author investigated more and obtained Theorem \ref{thm:1} and other theorems. The author would like to thank for Dominik Wrazidlo for studying the present paper and presenting new ideas on cobordisms of manifolds appearing as inverse images of regular values, which have motivated the author to study about the contents of this paper further. 
The author would like to thank for Takahiro Yamamoto for several interesting discussions on the present paper.

The author is a member of the project and supported by the project Grant-in-Aid for Scientific Research (S) (17H06128 Principal Investigator: Osamu Saeki)
"Innovative research of geometric topology and singularities of differentiable mappings"
( 
https://kaken.nii.ac.jp/en/grant/KAKENHI-PROJECT-17H06128/
).
\section{Reeb spaces of maps and triangulable maps}
\label{sec:2}

\begin{Def}
\label{def:1}
 Let $X$ and $Y$ be topological spaces. For $p_1, p_2 \in X$ and for a map $c:X \rightarrow Y$, 
 we define as $p_1 {\sim}_c p_2$ if and only if $p_1$ and $p_2$ are in
 the same connected component of $c^{-1}(p)$ for some $p \in Y$. The relation is an equivalence relation on $X$.
 We denote the quotient space by $W_c:=X/{\sim}_c$ and we call it the {\it Reeb space} of $c$.
\end{Def}

 We denote the induced quotient map from $X$ into $W_c$ by $q_c$. We can define $\bar{c}:W_c \rightarrow Y$ uniquely satisfying the relation $c=\bar{c} \circ q_c$.

\begin{Def}
\label{def:2}
Let $X$ and $Y$ be polyhedra. A continuous map $c:X \rightarrow Y$ is said to
 be {\it triangulable} if there exists a pair of triangulations of $X$ and $Y$ and the pair $({\phi}_X,{\phi}_Y)$ of homeomorphisms on $X$ and $Y$, respectively, such that the composition ${\phi}_Y \circ c \circ {{\phi}_X}^{-1}:X \rightarrow Y$ is a simplicial map with respect to the given triangulations. We also say that $c$ is triangulable with respect to $({\phi}_X,{\phi}_Y)$
\end{Def}

\begin{Prop}[\cite{hiratukasaeki}]
\label{prop:1}
For a triangulable map $c:X \rightarrow Y$ with respect to $(\phi_X,\phi_Y)$, the Reeb space $W_c$ is a polyhedron given by a homeomorphism ${\phi}_c$ from a polyhedron  and two maps $q_c:X \rightarrow W_c$
   and $\bar{c}:W_c \rightarrow Y$ are triangulable maps with respect to the corresponding pairs of the homeomorphisms.
\end{Prop}

Definition \ref{def:2} seems to be difficult to handle explicitly and we introduce another class of generic smooth maps.

\begin{Def}
\label{def:3}
A smooth map $c:X \rightarrow Y$ is said to be {\it Reeb-triangulable} if the Reeb space $W_c$ 
is regarded as a simplicial complex satisfying the following.
\begin{enumerate}
\item The inverse images of the interior of each simplex of dimension $\dim W_c$ of $W_c$ contains no singular point.  
\item The set of all points in $W_c$ whose inverse images have singular points forms a subcomplex of $W_c$.
\item The map $\bar{c}:W_c \rightarrow Y$ is triangulable with respect to a pair $(\Phi,\phi)$ of homeomorphisms
 where $\Phi$ makes $W_c$ the same simplicial complex as already given one and where $\phi$ makes $Y$ the same PL manifold as already given canonically.
\item $\dim W_c=\dim Y$.
\end{enumerate}
\end{Def}

{\it Stable} maps are essential smooth maps in higher dimensional versions of the theory of Morse functions or the theory of global
 singularity. 
\begin{Def}
\label{def:4}
A smooth map $c:X \rightarrow Y$ is said to be {\it topologically {\rm (}$C^{\infty}${\rm )} stable} if there exists an open neighborhood of $c$ in
 the space $C^{\infty}(X,Y)$ consisting of all smooth maps from $X$ into $Y$ given the Whitney $C^{\infty}$ topology and for any map $c^{\prime}$ in the neighborhood, there exists
 a pair $(\Phi,\phi)$ of homeomorphisms {\rm (}diffeomorphisms{\rm )} satisfying the relation $c^{\prime} \circ \Phi=\phi \circ c$.
\end{Def} 
 
As fundamental facts, Morse functions on smooth closed manifolds such that the values are
 always distinct at distinct two singular
  points are topologically and $C^{\infty}$ stable and a topologically or $C^{\infty}$ stable function is
    such a Morse function. Furthermore, such functions exist
   densely on any closed manifold. See \cite{golubitskyguillemin} for such fundamental facts.
   
  {\it Fold maps} are regarded as higher dimensional versions of Morse
   functions and at each singular point $p$, $f$ is
    of the form $(x_1,\cdots,x_m) \mapsto (x_1,\cdots,x_{n-1},\sum_{k=n}^{m-i(p)}{x_k}^2-\sum_{k=m-i(p)+1}^{m}{x_k}^2)$ for some
     integers $m,n$ and $i(p)$. For a smooth map, such a singular point is said to be a {\it fold} point and for a fold point $p$, $i(p)$ is taken as a non-negative integer not larger than $\frac{m-n+1}{2}$ uniquely: we call $i(p)$ the {\it index} of $p$. For a fold map, the set of all singular points of an index is a smooth submanifold of dimension $n-1$, and the
     restriction of the fold map to the singular set is an immersion and if the map is topologically or $C^{\infty}$ stable, then the immersion is transversal.  

We omit explanations of precise properties on (stable) fold maps in this paper. For such properties and explicit examples of fold maps, see \cite{saeki} and \cite{saeki2} and see also \cite{kitazawa} and \cite{kitazawa2}.

The following is known about {\it Thom} maps; the class of Thom maps is wider than that of topologically stable maps. A continuous map is said to be {\it proper} if the inverse image of each compact set in the target space is also compact.

\begin{Prop}[Shiota, \cite{shiota}]
\label{prop:2}
Proper Thom maps are always triangulable {\rm (}with respect to a pair of homeomorphisms giving the canonical triangulations of the smooth manifolds{\rm )}.
\end{Prop}

We omit the phrase  "with respect to a pair of homeomorphisms giving the canonical triangulations of the smooth manifolds" for such smooth maps in this paper.

We introduce some classes of smooth maps, give explanations on Reeb spaces and see that they are (Reeb-)triangulable. In presenting theorems and proofs later, we also present explicit Reeb spaces. 

\begin{Ex}
\label{ex:1}
\begin{enumerate}
\item
\label{ex:1.1}
Morse functions on closed manifolds are always triangulable and Reeb-triangulable; the Reeb spaces are regarded as graphs. More generally, for smooth functions on closed manifolds such that the sets of all singular values are finite sets, the Reeb spaces are also regarded as graphs. Such functions that are not (stable) Morse appear explicitly in \cite{sharko} and later in \cite{masumotosaeki} for example. 
\item
\label{ex:1.2}
A {\it simple} fold map is defined as a fold map such that each connected component of the inverse image of each singular value has only one singular point and they are triangulable and Reeb-triangulable (see \cite{saeki} for example). {\it Special generic} maps
 are defined as fold maps whose singular points are always of index $0$ and they are also simple fold maps. The Reeb space
  of a special generic map is regarded as a compact smooth manifold we can immerse into the target space. See \cite{saeki2} for an introduction to special generic maps and see also \cite{saekisakuma} and \cite{wrazidlo}. 
 Last, it is known that for a special generic map, the quotient map onto the Reeb space induces an isomorphism of homology groups whose degrees are not larger than the difference of the dimension of the source manifold and
 that of the target manifold. For a simple fold map such that inverse images of regular values are disjoint unions of spheres and satisfying suitable differential topological conditions on indices of fold points and differentiable structures of homotopy spheres appearing as inverse images of regular values (maps in Theorem \ref{thm:2} satisfy the condition), the quotient map onto the Reeb space induces an isomorphism of homology groups whose degrees are smaller than the difference of the dimension of the source manifold and
 that of the target manifold: see \cite{saekisuzuoka} and see also \cite{kitazawa} and \cite{kitazawa2}:.
\end{enumerate}
\end{Ex}

\begin{Rem}
The author does not know when a Reeb-triangulable smooth map is triangulable. It is true in some cases (see Example \ref{ex:1}) and it seems to be true for considerable classes of smooth maps. 
Moreover, it seems to be not so difficult to construct triangulable maps which are not Reeb-triangulable systematically. However, we do not consider construction of these maps in the present paper.
\end{Rem}
\section{A cobordism-like module generated by equivalence classes of smooth closed and connected manifolds}
\label{sec:3}

 We introduce a cobordism-like module generated by equivalence classes obtained from a suitable equivalence relation on a free module generated by all types of smooth, closed and connected manifolds. 

First, a {\it diffeomorphism type} is an equivalence class obtained by considering the equivalence relation on the family of smooth manifolds defined by diffeomorphisms between smooth manifolds.

  Let $k \geq 0$ be an integer. For a principle ideal domain $R$, we
   can define a free $R$-module ${\mathcal{N}}_k(R)$ (${\mathcal{O}}_k(R)$) generated by all (resp. oriented) diffeomorphism types of $k$-dimensional smooth, closed and connected (resp. oriented) manifolds which are mutually independent. 

 Let $f:M \rightarrow N$ be a triangulable smooth map from an $m$-dimensional smooth closed {\rm (resp. and oriented{\rm )} manifold of dimension $m$ into an $n$-dimensional smooth and oriented manifold. We define a submodule $A \subset {\mathcal{N}}_{m-n}(R)$ {\rm (}resp. ${\mathcal{O}}_{m-n}(R)${\rm )} satisfying several conditions.

We define a smooth compact submanifold {\it transverse} to the given smooth map.
A compact manifold $S$ smoothly embedded in $N$ is said to {\it transverse} to $f$ if for each point $a$ in $S-\partial S$ and each point $b \in f^{-1}(a)$, $df(T_b M) \oplus T_a S=T_a N$ and each point $a$ in $\partial S$ is a regular value of $f$: note that $f^{-1}(a)$ may be empty here. 

Let $\Gamma$ be a triangulation of $N$ given by an appropriate homeomorphism.
Let ${\Gamma}_f$ be the set of all connected components of inverse images of all $1$-dimensional smooth connected and compact submanifold with non-empty boundaries or closed interval smoothly embedded, transverse to the map $f$, having at most $1$ point of
 an ($n-1$)-dimensional face of a simplicial complex of $\Gamma$ and satisfying the following by $\bar{f}:W_f \rightarrow N$ satisfying $f=\bar{f} \circ q_f$.
\begin{enumerate}
\item The point in the face is not in any ($n-2$)-dimensional face.
\item The point in the face is in the interior of the curve in $N$.
\end{enumerate}
Let $\gamma \in {\Gamma}_{f}$ and let $A_{f,\gamma}$ be the set of all closed (resp. canonically oriented) connected submanifolds appearing as connected components of the inverse images of points in $\partial \gamma \bigcap {\rm Int} W_f$ by $f$. We can consider the sum of the products of $1$ or $-1$ and (resp. oriented) diffeomorphism types of all the submanifolds in $A_{f,\gamma}$ and denote the sum by $a_{f,\gamma}$. We define whether each coefficient of each diffeomorphism type is $1$ or $-1$ according to the orientation of the ($n-1$)-dimensional face induced from the $n$-dimensional oriented simplex the boundary point of the closed interval in the target belongs to. For example, see FIGURE \ref{fig:1} and in this case, the sum $a_{f,\gamma}$ is $a_1+a_2-a_3$ or $a_3-a_1-a_2$.

 Note that there may be a
 type appearing more than once and the coefficients appearing in $a_{f,\gamma}$ may not be the unit. For example, consider stable fold maps such that all the indices of singular points are $0$ or $1$, that the inverse images corresponding to regular values are disjoint unions of (oriented) standard spheres and that the differences of the dimensions of the source manifold and the target manifold are larger than $1$ or $m-n>1$, which are often considered and will be discussed also in this paper later in Theorem \ref{thm:2}. 
FIGURE \ref{fig:1} explicitly shows this case: as an explicit case, we can consider a case where two connected components of the three inverse images are diffeomorphic homotopy spheres being not standard spheres and the other is a standard sphere. In the case (with $R=\mathbb{Z}$), we can set the coefficients as $2$ and $-1$, respectively. 

\begin{figure}
\includegraphics[width=35mm]{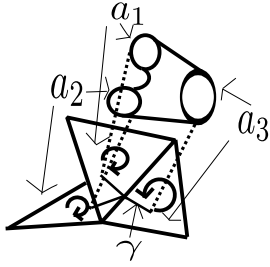}
\caption{The inverse image of a $1$-dimensional polyhedron $\gamma$ in the Reeb space by the quotient map $q_f$ (in this case the inverse image of $\gamma$ is PL homeomorphic to a manifold
 obtained by removing the interiors of three disjoint smoothly embedded closed discs of a dimension from the standard sphere of the same dimension).}
\label{fig:1}
\end{figure}

\begin{Def}
\label{def:5}
In the discussion just before, if for some triangulation $\Gamma$ of $N$ and each $\gamma \in {\Gamma}_{f}$, $a_{f,\gamma} \in A$ holds, then the submodule $A \subset {\mathcal{N}}_{m-n}(R)$ (resp. ${\mathcal{O}}_{m-n}(R)$) is said to be {\it {\rm (resp.} oriented{\rm )} compatible} with $f$.
\end{Def}
Note that the sign of each $a_{f,\gamma} \in A$ does not affect on the definition.

\begin{Def}
\label{def:6}
In the discussion here, we say that a smooth curve $\gamma$ in the target manifold $N$ is {\it simplicially transverse} to $f$ if $\gamma \in {\Gamma}_{f}$ holds for some simplicial complex $\Gamma$ of $N$ induced by an appropriate homeomorphism.
\end{Def}

\begin{Ex}
\label{ex:2}
Considering all {\rm (}oriented{\rm )} cobordism relations of the {\rm (}$m-n${\rm )}-dimensional smooth closed {\rm (}resp. and oriented{\rm )} manifolds, we can take $A$ so that ${\mathcal{N}}_{m-n}(\mathbb{Z}/2\mathbb{Z})/A$ {\rm (}resp. ${\mathcal{O}}_{m-n}(\mathbb{Z})/A${\rm )} is the ($m-n$)-dimensional smooth {\rm (}resp. oriented\rm {)} cobordism group. 
\end{Ex}

Note also that as FIGURE \ref{fig:2} shows, the inverse image of a point corresponding to a regular value in a simplex not in the interior ${\rm Int} W_f$ is zero in the module ${\mathcal{N}}_{m-n}(R)/A$ {\rm (}resp. ${\mathcal{O}}_{m-n}(R)/A{\rm )}$.

\begin{figure}
\includegraphics[width=35mm]{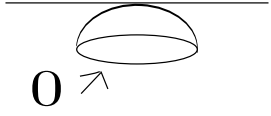}
\caption{The case of a simplex not in the interior of $W_f$.}
\label{fig:2}
\end{figure}


\begin{Ex}
Considering all {\rm (}oriented{\rm )} cobordism relations of the {\rm (}$m-n${\rm )}-dimensional smooth closed {\rm (}resp. and oriented{\rm )} manifolds, we can take $A$ so that ${\mathcal{N}}_{m-n}(\mathbb{Z}/2\mathbb{Z})/A$ {\rm (}resp. $H_n(W_f;{\mathcal{O}}_{m-n}(\mathbb{Z})/A){\rm )}$ is the ($m-n$)-dimensional smooth {\rm (}resp. oriented\rm {)} cobordism group. 
\end{Ex}

\section{The top homology groups of the Reeb spaces of triangulable maps}
\label{sec:4}      
\begin{Thm}
\label{thm:1}
  Let $f : M \rightarrow N$ be a triangulable smooth map or a Reeb-triangulable map from an $m$-dimensional smooth closed {\rm (}and oriented{\rm )} manifold $M$ into an $n$-dimensional smooth and oriented manifold $N$ satisfying $m>n \geq1$.

Let $R$ be a principle ideal domain and $A$ be an $R$-module {\rm (}resp. oriented{\rm )} compatible with $f$.
 If for a regular value $a$, there exists a connected component of $f^{-1}(a)$ such
 that the element obtained by the natural quotient map ${\mathcal{N}}_{m-n}(R)$ {\rm (}resp. ${\mathcal{O}}_{m-n}(R)${\rm )} onto ${\mathcal{N}}_{m-n}(R)/A${\rm (}resp. ${\mathcal{O}}_{m-n}(R)/A${\rm )} is not zero, then the homology group $H_n(W_f;{\mathcal{N}}_{m-n}(R)/A)$ {\rm (}resp. $H_n(W_f;{\mathcal{O}}_{m-n}(R)/A){\rm )}$ does not vanish.
\end{Thm}

It is regarded as an extension of the following proposition or the main result of  \cite{hiratukasaeki2} and
 we can prove Theorem \ref{thm:1} similarly. We denote the smooth oriented cobordism group of $k$-dimensional smooth closed and oriented manifolds by $\Omega_k$. 
 
\begin{Prop}[\cite{hiratukasaeki2}]
\label{prop:3}
Let {\rm (}the manifolds $M$ and $N$ be oriented and{\rm )} $f:M \rightarrow N$ be a smooth triangulable map. If for a regular value $a$, there exists a connected
 component of $f^{-1}(a)$ which is not {\rm (}resp. oriented{\rm )} null-cobordant, then for the Reeb space $W_f$, the homology group $H_n(W_f;\mathbb{Z}/2\mathbb{Z})$ {\rm (}resp. $H_n(W_f;{\Omega}_{m-n})${\rm )} is not trivial. 
\end{Prop}

Note that $R$ is $\mathbb{Z}$ and $\mathbb{Z}/2\mathbb{Z}$, respectively in this proposition.  

\begin{proof}[Proof of Theorem \ref{thm:1}.]
The proof is essentially same as that of Proposition \ref{prop:3}. For any $n$-dimensional simplex $\sigma$ of the polyhedron $W_f$, we are enough to take a chain $a_{\sigma} \sigma$ where  $a_{\sigma}$ is defined as the element obtained by the quotient map from ${\mathcal{N}}_{m-n}(R)$ onto ${\mathcal{N}}_{m-n}(R)/A$ (resp. ${\mathcal{O}}_{m-n}(R)$ onto ${\mathcal{O}}_{m-n}(R)/A$) from the (resp. oriented) diffeomorphism type of the inverse image of a point in $\sigma$ by $q_f:M \rightarrow W_f$. Note that this element does not depend on the point in the simplex by a discussion in the original paper or Lemma 3.1 of \cite{hiratukasaeki2}.

We are enough to take the sum of the elements for all the simplices in the simplicial complex. The resulting chain is a cycle by virtue of the definition of $A$ and the original proof. To see this, see FIGURE \ref{fig:1} again : the coefficient of each ($n-1$)-dimensional simplex of the boundary of the $n$-dimensional chain is zero ($a_1+a_2-a_3$ is in $A$ and by the natural quotient map it is sent to $0$). It is not a boundary since the dimension of $W_f$ is $n$. 
\end{proof}

\begin{Rem}
We do not need to assume that $N$ is oriented in the case where for each $r \in R$, $2r=0$ holds, for example. The statement of Proposition \ref{prop:3} for the case where the manifolds are not oriented gives an example of the cases where target manifolds are not oriented: in addition, for example, inverse images of regular values are also not oriented in these cases.  
\end{Rem}

\section{Modules compatible with given explicit maps and application of Theorem \ref{thm:1} to several classes of triangulable or Reeb-triangulable maps.}
First, we introduce several classes of (Reeb-)triangulable maps.
\begin{Def}
\label{def:7}
Let $f$ be a Reeb-triangulable smooth map. For the singular set $S(f)$, let $q_f(S(f))$ be a subcomplex of an appropriate simplicial
 complex $W_f$ and $W_f-{\rm Int} W_f \subset q_f(S(f))$. 
\begin{enumerate}
\item Suppose that $\dim (W_f-{\rm Int} W_f)=\dim W_f -1$ holds and that all points of $W_f-{\rm Int} W_f$ not being values
 of fold points of indice $0$ form a subcomplex of dimension smaller than $\dim (W_f-{\rm Int} W_f)$ : if the dimension is $0$, then we suppose that the subcomplex is empty. Suppose that the restriction of $q_f$ to the set of all
 singular points whose values are not in the subcomplex before is injective. Then $f$ is said to be {\it normal}.   
\item Suppose that $\dim q_f(S(f))=\dim W_f -1$ and that all points of $q_f(S(f))$ that are not values 
 of fold points form a subcomplex of dimension smaller than $\dim q_f(S(f))$ : if the dimension is $0$, then we suppose that the subcomplex is empty. Then $f$ is said to be {\it almost-fold}.
\item Suppose that the set
$\{p \in q_f(S(f)) \mid {q_f}^{-1}(p) \bigcap S(f)$ consists of more than $1$ singular points.$\}$ is a subcomplex of dimension smaller than $\dim q_f(S(f))$ : if
 the dimension is $0$, then we suppose that the subcomplex is empty. Then $f$ is said to be {\it stable-singular}. If the map is also almost-fold, then it is said to be {\it almost-stable-fold}.
\end{enumerate}
\end{Def}
Almost-fold maps are normal for example.
We see several examples.
\begin{Ex}
\begin{enumerate}
\item Proper ($C^{\infty}$ stable) {\it Morin} maps, which are regarded as generalizations of fold maps (see \cite{golubitskyguillemin} and see also \cite{morin} for example), are almost-fold ({\it resp}. almost-stable-fold). 
\item Proper $C^{\infty}$ stable maps are always almost-stable-fold in the case where the dimension of the target manifold is not larger than that of the source manifold and where the dimension of the target manifold is not so high (smaller than $6$). Note also that proper $C^{\infty}$ stable maps exist densely in such a case (and in addition, proper topologically stable maps exist densely in general). For these facts, see also \cite{golubitskyguillemin} for example and in the proof of Theorem \ref{thm:3} and FIGURE \ref{fig:4}, images of local maps of proper $C^{\infty}$ stable maps on manifolds whose dimensions are larger than $2$ into the plane will be presented.
\item Smooth functions on closed manifolds presented in the first part of Example \ref{ex:1} are Reeb-triangulable but they are generally not normal, almost-fold or stable-singular.
\end{enumerate}
\end{Ex}
\subsection{Cases of triangulable maps such that inverse images of regular values are highly connected.}
We denote the smooth oriented h-cobordism group of $k$-dimensional homotopy spheres by $\Theta_k$. For (oriented) h-cobordism groups of (homotopy spheres), see \cite{milnor} for example. 

\begin{Thm}
\label{thm:2}
Let $f:M \rightarrow N$ be an almost-stable-fold map from an $m$-dimensional smooth closed and oriented manifold $M$ into an $n$-dimensional smooth oriented manifold $N$ satisfying the following.
\begin{enumerate}
\item The relation $m>n \geq 1$ holds.
\item If $m-n$ is odd, then the index of each fold point is $0$ or $1$.
\item The inverse images of regular values are always disjoint unions of homotopy spheres.
\end{enumerate}
If for a regular value $a$, $f^{-1}(a)$ contains a connected component not diffeomorphic to a standard sphere, then $H_n(W_n;\Theta_{m-n})$ is not zero.
\end{Thm}

\begin{proof}
For a small curve simplicially transverse to $f$ which does not contain singular values, the inverse image is a disjoint union of cylinders of homotopy spheres. For a small curve simplicially transverse to $f$ containing just one singular value, by a fundamental discussion on the transversality of the curve to the singular value set, which we explain again and more precisely in the proof of Theorem \ref{thm:3}, an extension of the present theorem, we may regard that each connected component of the inverse image is regarded
 as a disjoint union of a standard closed disc of dimension $m-n+1$, a compact smooth manifold PL homeomorphic to a compact manifold obtained by removing $3$ disjoint smoothly embedded standard closed discs of dimension $m-n+1$ from $S^{m-n+1}$, or a cylinder of a homotopy sphere. Note that FIGURE \ref{fig:1} and FIGURE \ref{fig:2} show the cases where the connected components of the inverse images contain singular points.

From this argument and the definition of a module compatible with $f$, we can take a module $A$ compatible with $f$ as a commutative group so that the commutative group ${\mathcal{O}}_{m-n}(\mathbb{Z})/A$ is isomorphic to the group $\Theta_{m-n}$. There exists a one-to-one correspondence between the set of all oriented diffeomorphism types of ($m-n$)-dimensional oriented homotopy spheres and ${\Theta}_{m-n}$
 and by virtue of the fact, we have the desired result.
\end{proof}
\begin{Ex}
Let $m>n\geq 2$ hold. An $m$-dimensional manifold represented as the total space of a smooth bundle over $S^n$ with fibers diffeomorphic to a homotopy sphere $\Sigma$ obtained by gluing two standard closed discs by a diffeomorphism on the boundaries admits a stable fold map into ${\mathbb{R}}^n$ such that the following hold.
\begin{enumerate}
\item The singular value set consists of two spheres.
\item For connected components of the regular value set, the inverse images are empty set, $S^{m-n}$ and a disjoint union of two copies of $\Sigma$, respectively.
\end{enumerate}

See FIGURE \ref{fig:3}. Thus the Reeb space is simple homotopy equivalent to $S^n$ and the top homology group does not vanish, This fact of the top homology group follows also from Theorem \ref{thm:2} in the case where $\Sigma$ is a homotopy sphere which is not a standard sphere.
\begin{figure}
\includegraphics[width=40mm]{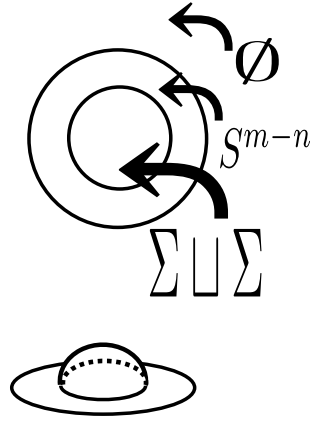}
\caption{A stable fold map on the total space of a smooth bundle over $S^n$ with fibers diffeomorphic to a homotopy sphere $\Sigma$ (each manifold represents the corresponding inverse image of the corresponding connected component of the regular value set) and the Reeb space.}
\label{fig:3}
\end{figure} 

For this, see also \cite{kitazawa} and \cite{kitazawa2}, in which such manifolds are characterized by maps presented here with additional differential topological conditions or more precisely, such maps obtained by the methods used in the proof of Theorem \ref{thm:4} or fundamental construction of {\it round} fold maps. 

Last, review about the fact on isomorphisms of homology and homotopy groups induced by the quotient maps onto the Reeb spaces mentioned in Example \ref{ex:1} (\ref{ex:1.2}). If the difference of the dimensions of the source and the target spaces is large, then the $n$-th homology group with coefficient ring $\mathbb{Z}$ of the source manifold is $\mathbb{Z}$, isomorphic to that of $S^n$. Note that in the case where the relation $m \neq 2n-1$ holds, such a fact holds. 
\end{Ex}

We remark on Theorem \ref{thm:2}.

\begin{Rem}
In \cite{kobayashisaeki}, stable maps of closed manifolds whose dimensions are larger than $2$ into the plane are discussed and for most of the maps in the paper, the 2nd homology groups (whose coefficient rings are $\mathbb{Z}$ or $\mathbb{Z}/2\mathbb{Z}$) of the Reeb spaces vanish; for example, the cases where the Reeb spaces are merely $2$-dimensional manifolds with non-empty boundaries are studied. 

In the last section of \cite{kobayashisaeki}, fold maps satisfying the assumption of Theorem \ref{thm:2} with the following additional conditions are studied (the explanation of the considered class here is a bit different from that of the original paper but essentially same as the original one).
\begin{enumerate}
\item The source manifold $M$ is simply-connected and the dimension of $M$ is larger than $3$.
\item The fold map $f:M \rightarrow {\mathbb{R}}^2$ is simple.
\item If the dimension of the source manifold is odd, then the index of each fold point is $0$ or $1$.
\item The inverse images of regular values are disjoint unions of standard spheres.
\item The top homology group or the 2nd homology group of $W_f$ is zero (the coefficient module is $\mathbb{Z}$).
\end{enumerate}
Let $A$ be a module over a principle ideal domain $R$.
 According to arguments by \cite{saekisuzuoka}, \cite{kitazawa} and \cite{kitazawa2}, the homology group $H_k(M;A)$ of source manifold and $H_k(W_f;A)$ is shown to be isomorphic for $0 \leq k \leq m-3$ (the proofs are done for homotopy groups and they are applicable for homology groups by fundamental properties of handle decompositions of PL manifolds). Thus $M$ is a homotopy sphere and $H_2(W_f;A)$ is zero. Thus, we can replace the term "disjoint unions of standard spheres" by "homotopy spheres" in the fourth
 additional assumption by virtue of Theorem \ref{thm:2}. Note also that we cannot deduce Theorem \ref{thm:2} directly from Proposition \ref{prop:3}.  
\end{Rem}
\begin{Rem}
In the case where the relation $m-n \neq 1,3,7,15$ holds, we can drop the assumption on indices in Theorem \ref{thm:2}. 
\end{Rem}

For integers $0<k_1<k_2$, the {\it $k_1$-connective $k_2$-dimensional smooth oriented cobordism group} $\mathcal{C}_{k_1}^{k_2}$ is a commutative group defined by an equivalence relation
 on the family of $k_2$-dimensional smooth, closed and $k_1$-connected oriented manifolds with an empty set: two elements in the family are equivalent if and only if their disjoint union bounds a $k_1$-connected smooth compact oriented manifold.
For the definition, see also \cite{stong} and \cite{wrazidlo} for example. We have the following as an extension of Theorem \ref{thm:2}.
\begin{Thm}
\label{thm:3}
Let $f:M \rightarrow N$ be an almost-stable-fold map from an $m$-dimensional smooth closed and oriented manifold $M$ into an $n$-dimensional smooth oriented manifold $N$ satisfying the following.
\begin{enumerate}
\item The relation $m-n \geq 2$ holds and $m-n \neq3,7,15 $.
\item $k$ is a positive integer satisfying $2k<m-n$ in the case where $m-n$ is even and $2k+1<m-n$ in the case where $m-n$ is odd and
 all manifolds appearing as connected components of inverse images of regular values are $k$-connected.
\end{enumerate}
If for a regular value $a$, $f^{-1}(a)$ contains a connected component which is not zero in $\mathcal{C}_{k}^{m-n}$ , then $H_n(W_f;\mathcal{C}_{k}^{m-n})$ is not zero.
\end{Thm}
\begin{proof}
For any curve simplicially transverse to $f$ containing just one singular value, if a connected component ${\gamma}_{\alpha}$ of ${\bar{f}}^{-1}(\gamma)$ is homeomorphic to a closed interval, then the inverse image of ${\gamma}_{\alpha}$ by the map $q_f$ is $k$-connected. We can know this by the fundamental discussions on handle decompositions of Morse functions
 or the attachments of ($m-n+1$)-dimensional handles with indices larger than $k$ and smaller than $m-n-k$. More precisely, each singular value correspond to a handle. Note that for this discussion, we may need to perturb the curve in the target manifold (see FIGURE \ref{fig:4}: see \cite{golubitskyguillemin} for fundamental theory of stable maps on manifolds whose dimensions are larger than $1$ into the plane including fundamental theory of {\it cusp} points etc.) and note also that in this case, the resulting curve may contain more than $1$ singular values in the interior and can be divided into curves simplicially transverse to the map each of which contains just one singular value with inverse image containing just one singular point by a finite number of points. It is also important that these perturbations do not affect the structure of a module compatible with the map if we take the module as a sufficiently small module.
\begin{figure}
\includegraphics[width=40mm]{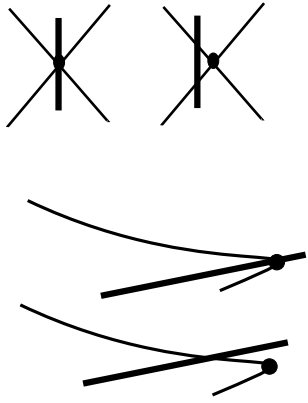}
\caption{Perturbations in cases where the target space is $2$-dimensional: The first figure represents a case where the curve contains a self-intersection of the singular value set and the second one represents a case where the curve contains a value of the map at a {\it cusp} point.}
\label{fig:4}
\end{figure} 

By virtue of this fact and Theorem \ref{thm:1}, we have the result.
\end{proof}
It seems to be difficult to find examples explaining Theorem \ref{thm:3} explicitly well due to the difficulty of knowing explicit structures of $\mathcal{C}_{k_1}^{k_2}$ in general.  

\begin{Rem}
In the case where the source manifolds are not oriented, by considering suitable coefficient modules, we can prove results similar to Theorem \ref{thm:2} and Theorem \ref{thm:3}. 
\end{Rem}

\subsection{Cases where the codimensions are $-2$.}
We consider cases where inverse images of regular values are closed surfaces.

\begin{Thm}
\label{thm:4}
Let $R \neq \{0\}$ be a principle ideal domain and $g$ be a positive integer larger than $2$. Then for any integer $m>2$, there exists a smooth, closed and connected, {\rm (}oriented{\rm )} manifold $M$ and a stable fold map $f:M \rightarrow {\mathbb{R}}^{m-2}$  
 satisfying the following.
\begin{enumerate}
\item There exists an $R$-module $A$ compatible with $f$ such that the value of the quotient map from ${\mathcal{N}}_{2}(R)$ {\rm (}resp. ${\mathcal{O}}_{2}(R)${\rm )}
 to ${\mathcal{N}}_{2}(R)/A$ {\rm (}resp. ${\mathcal{O}}_{2}(R)/A${\rm )} at the
 diffeomorphism type of a closed orientable {\rm (}resp. oriented{\rm )} surface ${\Sigma}_g$ of genus $g$ is not zero.
\item There exists a connected component of inverse image of a regular value diffeomorphic to ${\Sigma}_g$.
\item The number of connected components of the set of all singular points whose indices are $0$ are $2$ in the case where $m=3$ and $1$ in the case where $m>3$.
\end{enumerate}
Furthermore, $H_n(W_f;{\mathcal{N}}_{2}(R)/A)$ {\rm (}resp. $H_n(W_f;{\mathcal{O}}_{2}(R)/A)${\rm )} is not zero.
\end{Thm}
\begin{proof}
First we prove the former part in the case where $m=3$. We show in the case where $g$ is odd. We construct a Morse function such that the Reeb space is obtained by gluing two
 copies of a graph as presented in FIGURE \ref{fig:5-1} ; we attach one of the copies to the other one upside down and to show this, we draw a thick horizontal line in the bottom of the figure.  
We can see that the corresponding element of ${\mathcal{N}}_{2}(R)/A$ (${\mathcal{O}}_{2}(R)/A$) of the ({\it resp}. oriented) diffeomorphism type of ${\Sigma}_g$ does not vanish. 

In the case where $g$ is even, we use the graph presented in FIGURE \ref{fig:5-2} instead. Note that around the square-shaped dots, inverse images of regular values are closed and orientable surfaces whose genera are not smaller than $1$ and not larger than $g-2$.  

To show the statement for general $m \geq 3$, we consider the graph in the figure as the Reeb space of a Morse function on a compact manifold with non-empty boundary such that on the boundary the values are minimal and that the inverse image of the minimal value and the boundary coincide. We consider the product of the Morse function and the identity map ${\rm id}_{S^{n-1}}$. The
 image is regarded as the set of points in ${\mathbb{R}}^{m-2}$ such that the distances to the origin is not smaller than $a>0$ and not larger than $b>0$. We consider an appropriate product bundle over the set of points in ${\mathbb{R}}^{m-2}$ such that the distances to the origin is not larger than $a$. This is a fundamental construction of a {\it round} fold map introduced and systematically studied in \cite{kitazawa} and \cite{kitazawa2}. We will show another example of round fold maps later in FIGURE \ref{fig:7}.

\begin{figure}
\includegraphics[width=40mm]{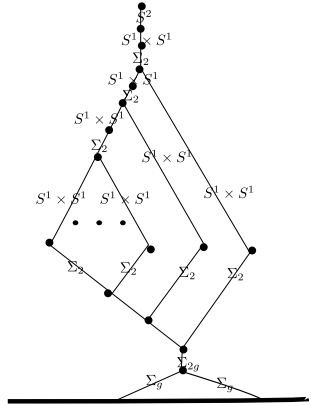}
\caption{The graph to obtain a Reeb graph for a desired Morse function: the manifolds represent corresponding connected components of inverse images of regular values, ${\Sigma}_k$ represents
 a closed orientable surface of genus $k>0$ and the dots represent corresponding (values at) singular points ($g$ is odd).}
\label{fig:5-1}
\end{figure} 
\begin{figure}
\includegraphics[width=40mm]{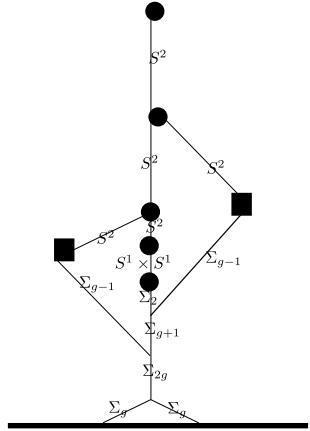}
\caption{The graph for the case where $g$ is even.}
\label{fig:5-2}
\end{figure} 
By the construction, we also have the third condition of the former part.

We have the latter part easily from Theorem \ref{thm:1} 
\end{proof}

Note that for the resulting maps we cannot apply Proposition \ref{prop:3}.
\begin{Thm}
\label{thm:5}
Let $f:M \rightarrow N$ be an almost-stable-fold map from an $m$-dimensional smooth closed manifold $M$ into an $n$-dimensional smooth oriented manifold $N$ satisfying the following.
\begin{enumerate}
\item $m-n=2$.
\item For a regular value $a$, $f^{-1}(a)$ contains a connected component diffeomorphic to a non-orientable closed surface.
\item The maximal value of genera of connected components of inverse images of regular values which are non-orientable, closed and connected surfaces is $g>0$.
\item For any small curve simplicially transverse to $f$ containing just one singular value such that the inverse image of the value contains just one singular point and that the point is fold, if a connected component ${\gamma}_{\alpha}$ of ${\bar{f}}^{-1}(\gamma)$ is homeomorphic to a closed interval, then the inverse images of points of the boundary by the map $q_f$ are as one of the following.
\begin{enumerate}
\item Both are closed and connected non-orientable surfaces of genera smaller than $g$.
\item Both are closed and connected orientable surfaces.
\item One is a closed and connected orientable surface and the other is a closed and connected non-orientable surface of genus $g$.   
\end{enumerate}
In this situation, if a connected component ${\gamma}_{\alpha}$ of ${\bar{f}}^{-1}(\gamma)$ is not homeomorphic to a closed interval, then it is homeomorphic to a Y-shaped $1$-dimensional polyhedron and the following hold.
\begin{enumerate}
\item The inverse images of points of the boundary by the map $q_f$ are all orientable or all non-orientable.
\item In the latter case just before, one of the non-orientable surfaces must be of genus $g$ and the genera of the others must be smaller than $g$. 
\end{enumerate}
\end{enumerate}
If for a regular value $a$, $f^{-1}(a)$ contains a connected component diffeomorphic to a non-orientable closed surface of genus
 smaller than $g$, then $H_n(W_f;\mathbb{Z})$ is not zero.
\end{Thm}
\begin{proof}
As in an argument of the proof of Theorem \ref{thm:3}, note that a curve simplicially transverse to the map can be perturbed appropriately and as a result we can regard this as a curve we can divide into curves simplicially transverse to $f$ each of which contains just one singular value such that the inverse image of the value contains just one singular point and that the point is fold. 

By virtue of all the assumptions, we can take a commutative group $A$ compatible with $f$ so that the group ${\mathcal{O}}_2(\mathbb{Z})/A$ is $\mathbb{Z}/2\mathbb{Z}$ or $\mathbb{Z}$. We can define $A$ so that the  ${\mathcal{O}}_2(\mathbb{Z})/A$ is generated by an element corresponding to a non-orientable closed and connected surface whose genus is smaller than $g$. More precisely, we are enough to set a generator of $A$ so that any element whose coefficient of the element $a$ corresponding to the non-orientable closed surface of genus $k<g$ appearing as a connected component of an inverse image of a regular value is not zero satisfies either of the following.
\begin{enumerate}
\item For the element $a_0$ corresponding to the non-orientable closed surface of genus $g$ and another element $b$ corresponding to the non-orientable closed surface of genus ${k}^{\prime}<g$ satisfying $k+k^{\prime}+1=g$, the element of the generator is represented as $a_0-a-b$. 
\item There exists another element $b$ corresponding to the non-orientable closed surface of genus ${k}^{\prime}<g$ satisfying $|k-k^{\prime}|=1$ and the element of the generator is represented as $a-b$. 
\end{enumerate}   
\end{proof}
As one of simplest examples, FIGURE \ref{fig:6} represents a Morse function satisfying the assumption of Theorem \ref{thm:5}.
FIGURE \ref{fig:7} represents a fold map such that the singular value set is concentrically embedded spheres as shown and that the corresponding inverse
 images are as denoted. It is an explicit fold map
 into the plane satisfying the assumption of Theorem \ref{thm:5}. Note also that it is an example of round fold maps : the presentation is announced in the proof of Theorem \ref{thm:4}. Note that this can be generalized to the cases for general dimensions according to a fundamental discussion as mentioned in the proof of Theorem \ref{thm:4}. Note that the figures represent
 maps to which we cannot apply Proposition \ref{prop:3}. 
\begin{figure}
\includegraphics[width=25mm]{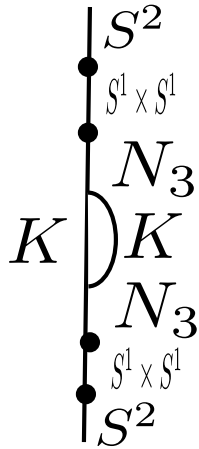}
\caption{An example of the Reeb spaces of Morse functions satisfying the assumption of Theorem \ref{thm:5} : $N_k$ is a non-orientable closed and connected surface of genus $k \geq 1$, $K$ is the Klein Bottle and its genus is $1$.}
\label{fig:6}
\end{figure} 
\begin{figure}
\includegraphics[width=35mm]{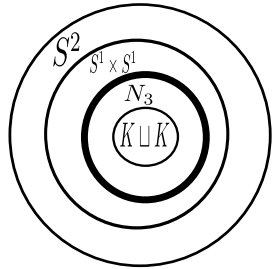}
\caption{An example of round fold maps satisfying the assumption of Theorem \ref{thm:5}.}
\label{fig:7}
\end{figure}

\end{document}